\documentclass[12pt,a4paper]{amsart}

\usepackage{amssymb,enumerate}

\usepackage{graphicx,amssymb,amsfonts,epsfig,amsthm,a4,amsmath,url,verbatim}
\usepackage[latin1]{inputenc}

\newtheorem{thm}{Theorem}[section]
\newtheorem{cor}[thm]{Corollary}
\newtheorem{lem}[thm]{Lemma}

\newtheorem{prop}[thm]{Proposition}

\theoremstyle{definition}
\newtheorem{defn}[thm]{Definition}

\newtheorem{que}[thm]{Question}

\newtheorem{rem}[thm]{Remark}

\numberwithin{equation}{section}

\newcommand{\N}{\mathbf{N}}
\newcommand{\Z}{\mathbf{Z}}
\newcommand{\R}{\mathbf{R}}

\newcommand{\Q}{\mathbf{Q}}

\newcommand{\tu}{\bigtriangleup}

\begin{document}
\title{On the space of ends of infinitely generated groups}
\author{Yves Cornulier}%
\address{CNRS and Univ Lyon, Univ Claude Bernard Lyon 1, Institut Camille Jordan, 43 blvd. du 11 novembre 1918, F-69622 Villeurbanne}
\email{cornulier@math.univ-lyon1.fr}
\subjclass[2010]{Primary 20F65; secondary 06E15, 20E08, 22F05}
\thanks{Supported by project ANR-14-CE25-0004 GAMME}




\date{June 10, 2019}

\maketitle
\begin{abstract}
We study the space of ends of groups. For a finitely generated group, this is a Cantor space as soon as it is infinite. In contrast, we show that for infinitely generated countable groups, it exhibits several behaviors. For instance, we show that for the free product $\Z\,\ast\,\Q$, it is a Cantor space, while for a free group of infinite rank, it is not metrizable. For arbitrary countable groups, we actually establish an alternative: the space of ends is either metrizable, or has a continuous map onto the Stone-\v Cech compactification of $\N$. We also show that the space of ends of a countable group has a continuous map onto the Stone-\v Cech boundary of $\N$ if and only if the group is infinite locally finite, and that otherwise it is separable.
For arbitrary groups, we also prove that the space of ends, if infinite, has no isolated point.

We also consider these questions for locally compact groups; also we extend Holt's theorem by showing that non-$\sigma$-compact regionally elliptic groups are 1-ended. 
\end{abstract}



\section{Introduction}

The study of ends of groups is a cornerstone in geometric group theory, and is mostly concerned with finitely generated groups, where it led to major results, such as Stallings' characterization of infinitely-ended groups and subsequent developments, such as the notion of accessibility, see for instance \cite{DD}. Here we rather focus on infinitely generated groups and phenomena that are specific to this setting. We start with a detailed state of the art, since such a survey cannot be found in the existing literature. After checking Definitions \ref{deftame} and \ref{defsends}, the reader can, in a first reading, directly read \S\ref{i_results} below, where the results of the paper are described.

\subsection{Background}

Freudenthal \cite{Fr1} introduced in 1931 the notion of ends of finitely generated groups. This is one of the earliest geometric notions on groups defined in such a generality. After Stone proved his duality theorem, Specker obtained a conceptual approach to ends of groups, notably allowing to define the space of ends for arbitrary groups.

Stone duality is a contravariant equivalence between the category of {\bf Stone spaces} (= Hausdorff, compact totally disconnected spaces), and the category of Boolean algebras. The forward functor maps a space $X$ to the algebra of continuous functions $X\to\Z/2\Z$. The backwards functor maps a Boolean algebra $A$ to its Stone space, namely its spectrum (the set of prime ideals of $A$, endowed with the compact topology induced by its closed inclusion into $\mathcal{P}(A)=2^A$, which is also the Zariski topology).

Let $G$ be a group. Then the left action $G$ on itself induces an action on its power set $\mathcal{P}(G)$ by Boolean algebra automorphisms, and this in turn induces an action on the quotient $\mathcal{P}^\star(G)$ of $\mathcal{P}(G)$ by the ideal of finite subsets. Denote by $\mathcal{P}^\star_G(G)$ the set of fixed points of this action, and let $\mathcal{P}_{(G)}(G)$ its inverse image in $G$. Thus, $\mathcal{P}_{(G)}(G)$ consists of the {\bf left $G$-commensurated} (also known as left almost invariant) subsets of $G$, that is, those subsets $A\subset G$ such that $\forall g\in G$, $gA\tu A$ is finite. 

\begin{rem}
The set of ends will be defined (Definition \ref{def_ends}) in terms of Stone duality. For the moment we start avoiding this definition, since some properties of the set of ends can be defined more elementarily.
\end{rem}

Note that the Boolean algebra $\mathcal{P}^\star_G(G)$ (or $\mathcal{P}^\star(G)$) is reduced to zero if and only if $G$ is finite; otherwise its unit (the image of the element $G$) is distinct from zero. The group $G$ is said to be {\bf 1-ended} if $\mathcal{P}^\star_G(G)$ is reduced to $\Z/2\Z$. Equivalently, this means that $G$ is infinite, and $\mathcal{P}_{(G)}(G)$ is reduced to finite and cofinite subsets.

The group $G$ is said to be {\bf finitely-ended} if $\mathcal{P}^\star_G(G)$ is finite, and {\bf infinitely-ended} otherwise. When (and only when!) $G$ is finitely-ended, the dimension $n$ of $\mathcal{P}^\star_G(G)$ as vector space over $\Z/2\Z$ is called the number of ends of $G$, and $G$ is called an {\bf $n$-ended} group. It is an early observation, independently due to Freudenthal \cite{Fr2} and Hopf \cite{Hop} in the finitely generated case, and Specker \cite{Sp} for arbitrary groups, that every group is either finite (0-ended), 1-ended, 2-ended, or infinitely-ended. Two-ended groups were characterized in these papers as infinite virtually cyclic groups. The characterization of infinitely-ended groups is Stallings' celebrated theorem in the finitely generated case \cite{Sta}. Its extension to the infinitely generated case makes a distinction between the case of locally finite groups, which deserves a specific study, and other groups. It leads to a clear-cut characterization of infinitely-ended groups \cite[Theorem IV.6.10]{DD}. The peculiar behavior of locally finite groups makes it convenient to summarize it in the next two classical theorems:

\begin{thm}
Let $G$ be a group that is not locally finite. Then (Freudenthal, Hopf, Specker) $G$ has 1, 2, or infinitely many ends. Moreover, in case $G$ is not 1-ended, one of the following holds:
\begin{enumerate}
\item (Freudenthal/Hopf, Specker) $G$ is 2-ended. This holds if and only if $G$ is virtually cyclic.
\item (Stallings, Dicks-Dunwoody) $G$ is infinitely-ended. This holds if and only if $G$ splits as an amalgam $A\ast_CB$ with $[A:C]\ge 2<[B:C]$ and $C$ finite, or as an HNN-extension $A\ast_C$ with $C$ a proper finite subgroup of $A$. Also, this holds if and only if $G$ admits an action by automorphisms on a tree, with finite edge stabilizers, such that there exist two elements acting as loxodromic isometries with no common endpoint.
\end{enumerate}
\end{thm}

For the locally finite case, it is worth introducing the following terminology: for a set $I$ we say that a function $f:G\to I$ is left $G$-almost invariant if for every $g\in G$, the set of $h\in G$ such that $f(gh)\neq f(h)$ is finite.

\begin{defn}\label{deftame} We say that $G$ is {\bf tamely-ended} if every left $G$-almost invariant function on $G$ has a finite image, and {\bf untamely-ended} otherwise.
\end{defn}

Although this is a very natural notion, I was not aware of its appearance in the literature at the time this work was carried out. The referee mentioned to me that it was very recently independently introduced by Banakh and Protasov \cite{BP} in the broader context of coarse spaces: $G$ is tamely-ended if and only if its left coarse structure is ``so-bounded" in the sense of \cite{BP}.

\begin{thm}
Let $G$ be a locally finite group. Then one of the following holds
\begin{enumerate}
\item if $G$ if finite, then it is 0-ended;
\item (folklore, see \S\ref{biends}) if $G$ is infinite countable, then it is infinitely-ended; actually, it is untamely-ended;
\item (Holt \cite{Hol}) if $G$ is uncountable, then it is 1-ended.
\end{enumerate}
\end{thm}

This seems to ``close" the classification. However, things become richer if one looks more closely the space of ends:

\begin{defn}[Specker \cite{Sp}]\label{defsends}
\label{def_ends}Let $G$ be a group. The {\bf space of ends} (or binary corona) of $G$ is the Stone dual $\mathcal{E}^\star_G(G)$ of $\mathcal{P}^\star_G(G)$ (which we now denote $\mathcal{E}^\star(G)$). The cardinal of $\mathcal{E}^\star(G)$ is called the {\bf number of ends} of $G$. The weight (the smallest cardinal of a basis of the topology) of $\mathcal{E}^\star(G)$ is just the dimension of $\mathcal{P}^\star_G(G)$ as a vector space over $\Z/2\Z$, which, when infinite, is just the cardinal of $\mathcal{P}^\star_G(G)$. We call it the {\bf end weight} of $G$; it is countable if and only if the space of ends is metrizable. (Beware that some authors confusingly refer to the end weight of $G$ as ``number of ends of $G$"; when finite they indeed coincide.)
\end{defn}

\begin{rem}
The space of ends is a space for which the set of points is more elaborate than its Boolean algebra of continuous functions to $\Z/2\Z$ (i.e., the Boolean algebra of clopen subsets). This is frequent in general topology: a well-known instance is the Stone-\v Cech compactification of $\N$, whose points are ultrafilters on $\N$, while its Boolean algebra of clopen subsets is simply the power set of $\N$.
\end{rem}

\begin{rem}
See Remark \ref{end_binary} for another formulation of the definition of space of ends avoiding an explicit reference to Stone duality.
\end{rem}

The quotient of $\mathcal{P}^\star_G(G)$ by its additive subgroup $\{0,1\}$ can be identified to the 1-cohomology of $G$ in the $G$-module of finitely supported functions $G\to\Z/2\Z$. This important remark is due to Specker \cite{Sp49}; see \cite{Coh} for details and extensions. Thus, the end weight of $G$ is equal to the dimension of this cohomology group. 

Furthermore, the right action of $G$ on itself induces an action of $G$ on $\mathcal{P}^\star_G(G)$, and hence on $\mathcal{E}^\star(G)$.

When $G$ is a infinitely-ended finitely generated group, things are well-understood: $\mathcal{P}_{(G)}(G)$ is countable, that is $\mathcal{E}^\star(G)$ is metrizable (see Remark \ref{fg_met} for a direct proof and Theorem \ref{caract_metri} for a generalization). Moreover, Freudenthal and Hopf independently \cite{Fr2,Hop} proved that $\mathcal{E}^\star(G)$ has no isolated point, and hence is a Cantor space (hence the number of ends is $2^{\aleph_0}$, and the end weight is $\aleph_0$).
 Moreover, Freudenthal \cite[\S 7]{Fr2} proved a result, which now is essentially interpreted as: the right action of $G$ on $\mathcal{E}^\star(G)$ is a convergence action (as formalized by Gehring and Martin \cite{GM87}). Stallings later obtained a splitting result for infinitely-ended groups. Combining with Stallings' theorem \cite{Sta}, Abels \cite{Ab77} deduced that the $G$-action on $\mathcal{E}^\star(G)$ is minimal (all orbits are dense).

 By definition, $G$ is untamely-ended if and only if there is a surjective left $G$-almost invariant function onto $\N$. Note that this implies that we can embed the power set of $\N$ into $\mathcal{P}^\star_G(G)$, and hence that $\mathcal{E}^\star(G)$ is non-metrizable (it has a surjective map onto the Stone-\v Cech compactification of $\N$); in particular, every finitely generated group is tamely-ended.




When $G$ is an infinite countable locally finite group, then it is easy to check that $\mathcal{E}^\star(G)$ is infinite; this was done in the abelian case by Specker in \cite{Sp}, and is mentioned as a routine fact in \cite[p.19-20]{Coh}. Indeed earlier Scott and Sonneborn \cite{ScS} proved that they have at least 2 ends, but their proof immediately shows that they are not tamely-ended. A more precise result is due to Protasov \cite{Pro11}; let us describe it.

A {\bf Parovi\v{c}enko space} is, informally speaking, a space that ``resembles" the Stone-\v{C}ech boundary of an infinite countable set. Formally speaking, it is a Stone space whose Boolean algebra has cardinal $2^{\aleph_0}$ and in which every nonempty $\mathrm{G}_\delta$ subset has nonempty interior. It is a basic but important observation that a Parovi\v{c}enko space has no isolated point, and more generally has no nonempty metrizable clopen subset.

The Stone-\v{C}ech boundary of an infinite countable set is such a space. Parovi\v{c}enko spaces are all homeomorphic if and only if the continuum hypothesis (CH) holds (see \cite[\S 1.2]{vM}).

\begin{thm}[Protasov \cite{Pro11}]
Let $G$ be an infinite countable locally finite group. Then $\mathcal{E}^\star(G)$ is a Parovi\v{c}enko space. In particular,
\begin{itemize}
\item $\mathcal{E}^\star(G)$ has no isolated point, and more generally has no nonempty clopen metrizable subspace;
\item if the continuum hypothesis holds, then $\mathcal{E}^\star(G)$ is homeomorphic to the Stone-\v{C}ech boundary of $\N$.
\end{itemize}
\end{thm}
This also implies that the end weight of $G$ is $2^{\aleph_0}$. Probably this latter easier fact is folklore: indeed the proof of \cite[Theorem 4]{ScS} that locally finite groups have at least 2 ends immediately adapts to prove this fact, embedding an infinite power set as a subalgebra of $\mathcal{P}_G^\ast(G)$, implying as well that the space of ends has cardinal $2^{2^{\aleph_0}}$.

\begin{rem}\label{BanaZari}
By Banakh-Zarichnyi \cite[Corollary 8]{BZ}, all infinite countable locally finite groups are coarsely equivalent. Hence, their spaces of ends are pairwise homeomorphic (regardless of CH). It is shown in \cite{BCZ} that it is consistent under ZFC (and the negation of CH) that it is not homeomorphic to the Stone-\v Cech boundary of $\N$.
\end{rem}

When $G$ is countable and locally finite, not much is known about the right action of $G$ on $\mathcal{E}^\star(G)$; note that this action is trivial if, for instance, $G$ is abelian. The space $\mathcal{E}^\star(G)$ is not separable: this follows from being Parovi\v{c}enko (Protasov's result), or, even simpler, because it has uncountable cellularity, which is one easy implication of  Theorem \ref{count_cell} below. Hence the action of $G$ on $\mathcal{E}^\star(G)$ cannot be minimal in this case.


It remains to consider the case of infinitely generated, non-locally-finite, infinitely-ended groups, especially countable ones. The only known result in this case I could locate are the following:


\begin{thm}[Abels, Satz 7.5 \cite{Ab77}]\label{abels75}
Let $G$ be an infinitely-ended group that is not locally finite. Then $\mathcal{E}^\star(G)$ has no finite $G$-equivariant quotient except singletons.
\end{thm}
In other words, the only finite orbits of the right $G$-action on the ring $\mathcal{P}_G^\star(G)$ are the singletons $\{0\}$, $\{1\}$. Abels also observes that this result fails for infinite countable locally finite groups, in which case there are arbitrary large finite quotients (i.e., arbitrary large finite invariant subsets in the Boolean algebra); which actually can be chosen with trivial action; this essentially follows from \cite[Theorem 4]{ScS}; see Proposition \ref{notame} for a precise statement.

Let us also mention the following result, which combines results of the given authors:

\begin{thm}[Freudenthal, Stallings, Dicks-Dunwoody, Abels]\label{nofior}
For every infinitely-ended non-locally finite group, the right action of $G$ on $\mathcal{E}^\star(G)$ has no finite orbit.
\end{thm}
(For an amalgam or HNN extension as given in the Stallings-Dicks-Dunwoody splitting theorem, Abels \cite{Ab77} proves that $G$ has no fixed point on its space of ends (relying on Freudenthal \cite{Fr2}). If there were a finite orbit, some finite index subgroup, whose space of ends is the same as $G$, would have a fixed point. At the time Abels wrote \cite{Ab77}, only the Stallings theorem was available, so the conclusion was only stated for finitely generated groups.)

\subsection{Results for discrete groups}\label{i_results}

For a group, it is natural to ask whether we have an alternative between having a metrizable space of ends, and the existence of a surjective almost left invariant map onto $\N$ (failure of being tamely-ended). For countable groups, this is indeed the case.

\begin{thm}[Metrizable vs untamely-ended alternative, \S\ref{tamalter}]
Let $G$ be a countable group. The following are equivalent:
\begin{enumerate}
\item the space of ends of $G$ is metrizable;
\item there is no continuous surjective map from the space of ends of $G$ onto the Stone-\v Cech compactification of $\N$;
\item $G$ is tamely-ended (Definition \ref{deftame}).
\end{enumerate}
\end{thm}

As already mentioned, finitely generated groups satisfy these conditions. The forward implications are actually straightforward. The theorem is actually  stated and proved in the much more general setting of $G$-sets, see \S\ref{tamalter}.

 Also the alternative fails for uncountable groups: for every uncountable cardinal $\alpha$ there exists a tamely-ended group $G$ of cardinal $\alpha$, with non-metrizable space of ends.

As already mentioned, finitely generated groups have a metrizable space of ends. However, for general countable groups, this is really an alternative. Indeed, the following is proved in \S\ref{s_both}.

\begin{thm}\label{exis}
There exist both infinitely generated countable groups with metrizable and non-metrizable space of ends. For instance,
\begin{itemize}
\item (see Proposition \ref{coupme}) the space of ends of a free group of infinite rank is non-metrizable (and actually it is not tamely-ended);
\item (see Proposition \ref{coume}) for every countable 1-ended group $H$, the space of ends of the free product $H\ast\Z$ is metrizable. For example, the space of ends of the free product $\Q\ast\Z$ is metrizable.
\end{itemize}
\end{thm}

The question whether we can characterize algebraically (e.g., in terms of splittings) non-tamely ended groups is natural, see Question \ref{istherechar}.


We also prove:

\begin{thm}[\S\ref{s_nois}]\label{noisola}
For every infinitely-ended group, the space of ends $\mathcal{E}^\star(G)$ has no isolated point. In particular, if metrizable, it is homeomorphic to a Cantor space.
\end{thm}

In the non-metrizable case, say for countable groups, can we obtain non-homeomorphic spaces? The answer is yes and we can indeed distinguish spaces of ends of locally finite groups. Recall that the {\bf cellularity} (or Suslin number) of a topological space $T$ is the upper bound of the set of cardinals $\alpha$ such that there is a family $(U_i)_{i\in\alpha}$ of pairwise disjoint nonempty open subsets. For $T$ infinite Hausdorff, it is infinite, and in addition there is an infinite family of pairwise disjoint nonempty open subsets. We say that the space has {\bf countable cellularity} if its cellularity is $\le\aleph_0$; for instance it holds for separable spaces.

\begin{thm}[Theorem \ref{count_cell2}]\label{count_cell}
Let $G$ be a countable group. The following are equivalent.
\begin{enumerate}
\item $G$ is infinite and locally finite;
\item $\mathcal{E}^\star(G)$ has a continuous map onto the Stone-\v Cech boundary of $\N$;
\item $\mathcal{E}^\star(G)$ has a family of $2^{\aleph_0}$ pairwise disjoint nonempty open subsets.
\item $\mathcal{E}^\star(G)$ does not have countable cellularity;
\item $\mathcal{E}^\star(G)$ is not separable.
\end{enumerate}
\end{thm}

Note that for a Stone space with Boolean algebra $A$, being separable is equivalent to the condition that $A=\{0\}$ or $A$ is isomorphic to a (unital) subalgebra of $2^\N$.

For uncountable groups, we have:

\begin{thm}[Theorem \ref{cell_uncount2}]\label{cell_uncount}
Let $G$ be an infinitely-ended group of uncountable cardinal $\alpha$. Then the cellularity of $\mathcal{E}^\star(G)$ is $\alpha$.


In particular, $\mathcal{E}^\star(G)$ is not separable (and hence not metrizable).
\end{thm}

\begin{cor}
For every infinitely-ended group $G$, the end weight of $G$ is $\ge |G|$.\qed
\end{cor}

\begin{rem}
For an infinitely-ended group $G$ of cardinal $\alpha$, the weight of $\mathcal{E}^\star(G)$, i.e. the cardinal of $\mathcal{P}_G^\star(G)$, thus belongs to $[\alpha,2^\alpha]$. For each infinite cardinal $\alpha$, each bound is achieved by some group, by Proposition \ref{coupme} (upper bound $2^\alpha$) and Proposition \ref{coume} (lower bound $\alpha$). 
\end{rem}



\subsection{Some open questions}

It follows from the previous results that we can split the class of countable groups into six classes $\mathcal{M}_0$, $\mathcal{M}_2$, $\mathcal{M}_1$, $\mathcal{M}_\infty$, $\mathcal{X}$, and $\mathcal{U}$, where
\begin{itemize}
\item $\mathcal{M}_i$ is the class of countable groups for which the space of ends is metrizable and has $i$ points (when $i=\infty$, this is a Cantor space by Theorem \ref{noisola}).
\item $\mathcal{U}$ is the class of countable groups for which the space of ends is not separable. By Theorem \ref{count_cell}, it coincides with the class of infinite countable locally finite groups. By the Banakh-Zarichnyi theorem (see Remark \ref{BanaZari}), they provide a single homeomorphism type of space of ends.
\item $\mathcal{X}$ is the remaining class, which, by Theorem \ref{count_cell}, consists of those countable groups for which the space of ends is separable but not metrizable (countable is redundant by Theorem \ref{cell_uncount}). The class $\mathcal{X}$ notably contains free products of infinite sequences of nontrivial groups.
\end{itemize}
Among these six classes, all but $\mathcal{X}$ yield a unique space of ends up to homeomorphism. Those in $\mathcal{X}$ provide non-metrizable, separable Stone spaces with no isolated points, but it is not clear if this yields a single homeomorphism type, or many different ones.

\begin{que}
How many homeomorphism types do there exists of spaces of ends of countable groups in the class $\mathcal{X}$? (the answer is a cardinal in $[1,2^{\aleph_0}]$.)
\end{que}

This question can be specified: for instance, I do not know whether $\Z^{\ast \N}$ (free group on countably many generators) and the countable free product $(\Z^2)^{\ast \N}$ have homeomorphic space of ends.

By the Stallings-Dicks-Dunwoody theorem, the class $\mathcal{M}_2\sqcup\mathcal{M}_\infty\sqcup\mathcal{X}$ consists of those countable groups that split over a finite subgroup. It is natural to ask whether there is a similar characterization of the class $\mathcal{X}$.

\begin{que}\label{istherechar}
Is there a characterization of groups in the class $\mathcal{X}$, among countable groups, in terms of splittings over finite groups (or equivalently in terms of actions on trees with finite edge stabilizers)?
\end{que}

A rough expectation is that $\mathcal{X}$ should be related to infinite splittings over finite subgroups. Let me ask anyway a more precise question in the torsion-free case:

\begin{que}
Let $G$ be a countable torsion-free group. Is it true that $G$ is metrizably-ended if and only if it has finite Kurosh rank (i.e.,  splits as a free product of finitely many freely indecomposable groups)?
\end{que}

Groups of infinite Kurosh rank can be fairly complicated: indeed Kurosh \cite{Kur} constructed a countable, torsion-free group that does not split as a (finite or infinite) free product of freely indecomposable groups, namely the group $\Gamma$ with presentation
\[\langle a_n:n\ge 0,\,b_n:n\ge 1\mid a_nb_na_n^{-1}b_n^{-1}=a_{n-1}:n\ge 1\rangle.\]
Indeed $\Gamma$ is isomorphic to $\Gamma\ast\Z$ and has the property that all its freely indecomposable subgroups are cyclic, but is not free, because it is not residually nilpotent.


One more natural notion is that of cofinality, in the following sense: say that a Boolean algebra has countable cofinality if it can be written as a strictly increasing sequence of subalgebras, and uncountable cofinality otherwise. For instance, a countable Boolean algebra has countable cofinality if and only if it is infinite. This notion was notably studied by Koppelberg \cite{Kop}; see also \cite[\S 12]{VDB} (which also discusses interesting related cofinality notions). For the Stone space, it corresponds to the property of being an inverse limit of a sequence of proper quotients. A related notion for the Boolean algebra is the property of having an infinite countable quotient (which implies countable cofinality). For the Stone space, this means the existence of an infinite, closed metrizable subset, or, equivalently, of a converging injective sequence. 

\begin{que}For which groups $G$ does the Boolean algebra $\mathcal{P}^\star_G(G)$ (resp.\ $\mathcal{P}_{(G)}(G)$) have countable cofinality? For which groups does it admit an infinite countable quotient? 
\end{que}

Note that the latter question is equivalent to asking whether the space of ends $\mathcal{E}^\star(G)$ has an injective converging sequence.

For instance, if $G$ is a countable locally finite group then $\mathcal{P}^\star_G(G)$ has uncountable cofinality; this follows from the previous description and \cite[Theorem 1]{Kop}. In general, I do not know the answer for those countable groups with non-metrizable and separable space of ends. 
The referee suggested that the answer should be positive in the case of a free group of infinite countable rank (which was explicitly asked in a preliminary version), by embedding the space of ends of a finitely generated free subgroup. This is indeed the case, by virtue of the following:

\begin{prop}[See Proposition \ref{freepro_inj}]\label{free_prod_inj_intro}
Let $G$ be a group with a free product decomposition $H\ast L$ where the space of ends of $H$ is a Cantor space. Then the induced map from $\mathcal{E}^\star(H)$ to $\mathcal{E}^\star(G)$ is injective, and in particular the space of ends of $G$ includes a Cantor space and the Boolean algebra $\mathcal{P}_{G}^\star(G)$ has a infinite countable quotient (namely $\mathcal{P}_{H}^\star(H)$).
\end{prop}



\subsection{Bi-ends}

Next, let us mention a forgotten notion of ends, somewhat more elementary than the usual notion, but which is worth developing here. It consists in considering groups $G$ as $G\times G$-sets, rather than $G$-sets under left multiplication. This study appeared in work \cite{ScS} of Scott and Sonneborn (1963), who did not make the connection with the theory of ends.

 The notion of $G$-commensurated subset makes sense in an arbitrary $G$-set; here the group is $G\times G$ and the set is $G$ under left-and-right multiplication. Thus, we call a bi-commensurated subset of $G$ a subset $X$ such that $gX\tu X$ and $Xg\tu X$ are finite for every $g\in G$. The set $\mathcal{P}_{(G\times G)}(G)$ of bi-commensurated subsets is a Boolean subalgebra of $\mathcal{P}_{(G)}(G)$. Its quotient by the ideal of finite subsets is denoted $\mathcal{P}_{G\times G}^\star(G)$, is a Boolean subalgebra of $\mathcal{P}_{G}^\star(G)$, namely it precisely consists of the fixed points of the right $G$-action. Thus the Stone dual $\mathcal{E}^\star_{G\times G}(G)$ of $\mathcal{P}_{G\times G}^\star(G)$ is a quotient of $\mathcal{E}^\star(G)$; we call it the {\bf space of bi-ends} of $G$. (It is actually the quotient of $\mathcal{E}^\star(G)$ by the right $G$-action in the category of Stone spaces.) We then call the cardinal of $\mathcal{E}^\star_{G\times G}(G)$ the number of bi-ends of $G$; it is bounded above by the number of ends.

Let us now state the theorem, since the statement has not been written down, and insisting that it does not rely on Stallings' theorem. For a property $P$, recall that a group is $P$-by-$Q$ if it has a normal subgroup with Property P such that the quotient satisfies Property Q. It is known, and originally observed by Wall \cite[Lemma 4.1]{Wa}, that a $\Z$-by-finite (=infinite virtually cyclic) group is either finite-by-$\Z$ or finite-by-$D_\infty$, where $D_\infty$ denotes the infinite dihedral group. These 2-ended groups turn out to be distinguished by bi-ends.



\begin{thm}[following from Freudenthal, Scott-Sonneborn]\label{tbiends}
Let $G$ be a group, viewed as a $G\times G$-set under left-right translation. Then
\begin{enumerate}
\item\label{scs1} If $G$ is finite, it has 0 bi-end;
\item\label{scs2} if $G$ is finite-by-$\Z$, it has 2 bi-ends;
\item\label{scs3} if $G$ is countable and locally finite, then it has $2^{2^{\aleph_0}}$ bi-ends (and actually its space of bi-ends has a continuous surjection onto the Stone-\v Cech boundary of $\N$);
\item\label{scs4} otherwise, $G$ has a single bi-end. Thus ``otherwise" means that one of the following holds:
\begin{itemize}
\item $G$ is finitely generated and not virtually cyclic, or
\item $G$ is finite-by-$D_\infty$;
\item $G$ is countable, infinitely generated and not locally finite, or
\item $G$ is uncountable.
\end{itemize}
\end{enumerate}
\end{thm}

See \S\ref{biends} for the proof and a discussion.
Note that, in some sense, Theorem \ref{tbiends} shows that the theory of bi-ends ``wipes out" almost all the richness of the theory of infinitely-ended groups and its connection with splittings.


\subsection{Results for locally compact groups}

The transition of the history of ends of groups from papers in German language by Freudenthal, Hopf, Specker, and finally Abels, to Stallings and its successors in English language, has a noticeable aftermath in the next 30 years: the initial framework of topological groups has been narrowed to discrete groups (or, separately, to connected Lie groups, for which the infinitely-ended case does not occur). 

Let $G$ be an {\bf LC-group} (locally compact group). First, let us start by giving the relevant definitions in this context, following Specker \cite{Sp}. For $G$ an LC-group (locally compact group), we do not exactly copy the definition from the discrete case, because we need a little uniformity in the definition of $G$-commensurated subset. Namely, $\mathcal{P}_{(G)}(G)$ is defined as the set of subsets $X$ of $G$ such that $KX\smallsetminus K$ has compact closure for every compact subset $K$ (call them topologically left-commensurated subsets). This is a Boolean subalgebra of $\mathcal{P}(G)$, containing the ideal of subsets with compact closure, and its image modulo this ideal is denoted as $\mathcal{P}_G^\star(G)$. The Stone dual $\mathcal{E}^\star(G)$ of the latter is the space of ends of $G$. The right action of $G$ induces a continuous action of $G$ on $\mathcal{E}^\star(G)$.

The most basic results extend : the space of ends has 0, 1, 2, or infinitely many points \cite{Sp}. When $G$ is compactly generated, it is metrizable \cite{Abe}; if moreover $G$ is infinitely-ended, it is homeomorphic to a Cantor space \cite{Ab77}. That it has no isolated point actually holds without the compact generation assumption:

\begin{thm}\label{noisola2}
For every infinitely-ended LC-group, the space of ends $\mathcal{E}^\star(G)$ has no isolated point.
\end{thm}

One difficulty in proving this, is that in the context of locally compact groups there is an additional class of groups to consider: focal infinitely-ended LC-groups. Namely, consider a compact group $K$ and a continuous isomorphism $\theta$ of $K$ onto a proper (open) subgroup of finite index. The resulting HNN-extension is called a {\bf focal infinitely-ended LC-group}. Note that such an LC-group cannot be discrete. (Some more focal LC-groups are considered in \cite{CCMT}, which are 1-ended. In \cite{CCMT}, focal infinitely-ended groups are called focal groups of totally disconnected type.)

As observed by Abels \cite{Ab77}, the right action of a focal infinitely-ended LC-group $G$ on $\mathcal{E}^\star(G)$ has a fixed point. Actually, in this case, this fixed point $\omega$ is unique, and the action of $G$ on $\mathcal{E}^\star(G)\smallsetminus\{\omega\}$ is transitive \cite[Lemma 3.4]{CCMT}. Indeed, in this case, one can identify $\mathcal{E}^\star(G)$ to the boundary of the Bass-Serre tree of the HNN-extension $K\ast_\theta$, which is a regular tree of valency $1+[K:\theta(K)]$.



The Stallings splitting theorem (for finitely generated groups) has both been generalized to compactly generated locally compact groups by Abels \cite{Abe}, and to infinitely generated non-locally-finite groups by Dicks-Dunwoody, but the generalization to non-discrete locally compact groups that are not compactly generated is still conjectural. Therefore, the analogue of Theorem \ref{nofior} (where one should exclude the focal case) is also conjectural, and holds if the splitting theorem holds in full generality, since Abels' proof \cite{Ab77} that there is no finite orbit in the case with splittings, was written in the locally compact setting. The same applies to the analogue of Theorem \ref{cell_uncount} (where $\alpha$ is meant to be the cardinal of $G/K$ for some/any $\sigma$-compact open subgroup $K$).


Recall that a locally compact group is {\bf regionally elliptic} (or locally elliptic) if each compact subset is contained in some compact open subgroup.
Abels \cite[7.10]{Abe} asked whether there exists a non-$\sigma$-compact regionally elliptic, locally compact group with infinitely many ends. Scott and Sonneborn \cite{ScS} had solved negatively this for discrete (locally finite) abelian groups; Holt \cite{Hol} solved negatively the question for arbitrary discrete (locally finite) groups, and elaborating on Holt's proof, we answer Abels' question in the general case:


\begin{thm}[see \S\ref{s_holt}]\label{HLC}
Every non-$\sigma$-compact, regionally elliptic locally compact group is 1-ended.
\end{thm}

Let us also mention that the analogue of Protasov's theorem to the locally compact setting ($\sigma$-compact regionally elliptic groups) holds, and actually the proof follows from Protasov's results, so we omit the proof.







\subsection{Space of ends as a metric space}



As we have mentioned, the topological classification of spaces of ends of finitely generated groups is completely settled. However, there is a finer metric structure on the space of ends, which is worth pointing out.

Namely, in a group $G$ endowed with a word length $|\cdot|$ with respect to a finite generating subset, call the radius of a nonempty finite subset $F$ the supremum $\max_{x\in F}|x|$. Say that ends $\omega,\omega'$ are separated by $F$ if they are ends of distinct components of $G\smallsetminus F$. Define $D(\omega,\omega')$ as the infimum of the radii of finite subsets separating $\omega$ and $\omega'$ (0 if $\omega=\omega'$). Finally define $d(\omega,\omega')=\exp(-D(\omega,\omega'))$. This is a metric on the space of ends, defining its topology. For two choices of word lengths, the identity map of the space of ends is a bi-H\"older homeomorphism between the resulting metric spaces. A natural question is then

\begin{que} Do there exist two finitely generated infinitely-ended groups for which the space of ends are not bi-H\"older equivalent?
\end{que}


\section{Preliminaries}

We need the following well-known fact (see \cite[Proposition 4.B.2]{CorFW} for a more general result).
\begin{lem}[Lemma 2.3 in \cite{Coh}]\label{actfg}
Let $G$ be a group, $H$ a finitely generated subgroup, and $X$ a left $H$-commensurated subset of $G$. Then there exists a finite subset $F$ of $G$ such that for every right coset $Hg$ disjoint from $F$, we have $X\cap Hg\in\{\emptyset,Hg\}$.
\end{lem}
\begin{proof}Since we generalize the result below, we include the easy proof.
Let $S$ be a finite generating subset of $H$. For $s\in S$ and $Hg\in H\backslash G$ (quotient of $G$ by the left action of $H$), write $X_g=X\cap Hg$. Then we have $sX\tu X=\bigsqcup_{g\in H\backslash G}sX_g\tu X_g$. Since this is finite and $S$ is also finite, we deduce that for all but finitely many right cosets $Hg$, we have $sX_g\tu X_g=\emptyset$ for all $s\in S$, which means that $X_g$ is left $H$-invariant.
\end{proof}

Here is the locally compact version of this lemma; the proof being an immediate adaptation of the above one:
\begin{lem}\label{actfg2}
Let $G$ be a locally compact group, $H$ an open, compactly generated subgroup, and $X$ a topologically left $H$-commensurated subset of $G$. Then there exists a finite subset $F$ of $G$ such that for every right coset $Hg$ disjoint from $F$, we have $X\cap Hg\in\{\emptyset,Hg\}$.\qed
\end{lem}

The notion of {\bf $G$-commensurated subset} makes sense for an arbitrary $G$-set $X$. They form a Boolean subalgebra $\mathcal{P}_{(G)}(X)$ of $\mathcal{P}(X)$, and its quotient by the ideal of finite subsets of $X$ is denoted by $\mathcal{P}^\star_G(X)$. The corresponding Stone spaces are called the {\bf end compactification} and the {\bf space of ends} $\mathcal{E}^\star(X)$ of the $G$-set $X$. Ends of groups are the particular case of simply transitive actions. Also, the end compactification of a trivial action is just the Stone-\v Cech compactification of the given set.

\begin{rem}\label{end_binary}
Let $G$ be a group and $X$ a $G$-set. Embed $X$ into the compact space $\{0,1\}^{\mathcal{P}_{(G)}(X)}$ by mapping $x$ to $\iota(x)=(\mathbf{1}_M(x))_{M\in\mathcal{P}_{(G)}(X)}$. Then the end compactification of the $G$-set $X$ can be identified to the embedding of $X$ into the closure of $\iota(X)$. This point of view is the one used by Protasov in \cite{Pro03}. Then $\iota(X)$ is open in its closure, and the complement identifies to the space of ends of the $G$-set $X$.
\end{rem}

\section{The alternative metrizably-ended / untamely-ended}\label{tamalter}
\begin{thm}\label{setcoume}
Let $G$ be a countable group and $X$ a $G$-set. The following are equivalent:
\begin{enumerate}
\item\label{it_mt1} $X$ is not metrizably-ended;
\item\label{it_mt2a} the space of ends of $X$ has a closed subset homeomorphic to the Stone-\v Cech compactification of $\N$;
\item\label{it_mt2} the space of ends of $X$ has a continuous surjective map onto the Stone-\v Cech compactification of $\N$;
\item\label{it_mt3} the end compactification of $X$ has a continuous surjective map onto the Stone-\v Cech compactification of $\N$;
\item\label{it_mt4} $X$ is not tamely-ended.
\end{enumerate}
\end{thm}
\begin{proof}
We start with the easy $\Leftarrow$ implications: if $X$ is not tamely ended (\ref{it_mt4}), consider a surjective almost $G$-invariant map $X\to\N$. Then inverse image map is an injective Boolean algebra homomorphism from $2^\N$ into the Boolean algebra of $G$-commensurated subsets of $X$, which by Stone duality yields a map in (\ref{it_mt3}). 

If there is such a map on the end compactification (\ref{it_mt3}), then by Stone duality we deduce an injective map from the Boolean algebra $2^\N$ into the Boolean algebra of $G$-commensurated subsets of $X$. Considering a partition of $\N$ into infinitely many infinite subsets, we can suppose that its image intersects the ideal of finite subsets is reduced to $\{0\}$, and hence the the Boolean algebra of commensurated subsets modulo finite subsets contains a copy of $2^\N$, which yields by Stone duality a map as in (\ref{it_mt2}).

Suppose that (\ref{it_mt2}) holds (for an arbitrary Stone space $E$): $E$ has a continuous surjective map $f$ to the Stone-\v Cech compactification $\beta\N$. Lifting each element of $\N$, we obtain a map $\N\to E$, which thus extends to a unique continuous map $u:\beta\N\to E$, such that $f\circ u$ is the identity on $\N$. By density, $f\circ u$ is the identity.

If (\ref{it_mt2a}) holds, then the Boolean algebra of commensurated subsets of $X$ is uncountable, that is, (\ref{it_mt1}) holds.

It remains to prove the less trivial implication, namely that (\ref{it_mt1}) implies (\ref{it_mt4}). Suppose that $X$ is not metrizably-ended. If $X$ has infinitely many $G$-orbits, then $X$ is clearly not tamely-ended: indeed if $Y=G\backslash X$ is the set of orbits, then the projection map $X\to Y$ is $G$-invariant.

Hence we can suppose that $X$ has finitely many orbits, and then we immediately reduce to the case when $G$ acts transitively on $X$, namely $X=G/H$, after choice of a base-point $o$ in $X$.

Write $G$ as an ascending union of finite subsets: $G=\bigcup F_n$ with $F_n\subset F_{n+1}$ for every $n$, and let $G_n$ be the subgroup generated by $F_n$. For any given discrete abelian group $A$, denote by $M_G(X,A)$ the set of almost $G$-invariant functions $X\to A$. (The group structure of $A$ does not affect its definition, but simply makes $M_G(X,A)$ an abelian group.) For a subgroup $L$ of $G$, denote by $M_G^L(X,A)$ its subgroup of $L$-invariant functions. Then $M_G^{G_n}(X,A)$ is the kernel of the homomorphism from $M_G(X,A)$ to $K(X,A)^{F_n}$ mapping $f$ to $(f-g\cdot f)_{g\in F_n}$, where $K(X,A)$ is the group of finitely supported functions $X\to A$.

We now suppose that $A=\Z$. Then $K(X,A)^{F_n}$ is countable, and since $X$ is not metrizably-ended, $M_G(X,A)$ is uncountable. Hence the kernel $M_G^{G_n}(X,A)$ is uncountable; in particular it is not reduced to constant maps. Post-composing by a self-map of $\Z$, there exists an element of $f_n$ of $M_G^{G_n}(X,A)$ such that $f_n(o)=0$ and $f_n(X)=\{0,n\}$. Write $X_n=G_no=G_n/(H\cap G_n)\subset X$, so $X$ is the ascending union of the subsets $X_n$. Then $f_k$ is $0$ on $X_n$ for all $k\ge n$; in particular, the infinite sum $f=\sum_n f_n$ is a well-defined function on $X$. Since $f_n\ge 0$ and $f_n$ takes the value $n$, the function $f$ is unbounded, and in particular $f$ has an infinite image. 

Finally, we check that $f$ is almost $G$-invariant. Fix $g\in G$. So $g\in G_n$ for some $n$. Let $F$ be the set of $x\in X$ such that for some $k<n$ we have $f_k(gx)\neq f_k(x)$. Then $F$ is finite. For $k\ge n$, $f_k$ is $G_n$-invariant. Hence $f(gx)\neq f(x)$ for all $x\notin F$, so $f$ is almost $G$-invariant. Thus $X$ is not tamely-ended.
\end{proof}

The proof actually shows the following:

\begin{thm}\label{caract_metri}
Let $G$ be a countable group.
Let $X$ be a $G$-set. Then $X$ is metrizably-ended if and only if $X$ has finitely many $G$-orbits and there exists a finitely generated subgroup $H$ of $G$ such that the only $H$-invariant, $G$-commens\-urated subsets of $X$ are $G$-invariant (which are finitely many).\qed
\end{thm}

\begin{rem}\label{fg_met}
As a particular case of Theorem \ref{caract_metri}, we obtain the classical fact that if $G$ is finitely generated (say by some finite subset $S$), then its space of ends is metrizable. Here the proof simplifies as follows. Let $M(G)$ be the set of almost $G$-invariant functions $G\to\Z/2\Z$, and $K(G)$ those finitely supported ones. Map $M(G)$ to $K(G)^S$ by $f\mapsto \partial f=(f-s\cdot f)_{s\in S}$. Then the kernel of $\partial$ is reduced to the two constant functions. Since $K(G)^S$ is countable, it follows that $M(G)$ is countable. This means that the end compactification is metrizable.
\end{rem}

\begin{rem}
Let $X$ be the product of uncountably many finite sets, each of cardinal $\ge 2$). Then $X$ is not metrizable, but has no continuous surjective map onto the Stone-\v  Cech compactification of $\N$.
\end{rem}

\begin{rem}
The space of ends can be defined in the context of coarse spaces. One instance is the coarse structure on a set $X$ induced by a group action. Protasov \cite{Pro08} proved that if a coarse structure has bounded geometry, then it is coarsely equivalent to the coarse structure induced by some group action. If in addition the underlying set is countable, the acting group can be chosen to be countable as well (this is a consequence of the proof). Therefore, since coarse equivalences between coarse spaces induce homeomorphisms between space of ends, it follows (combining Protasov's result with Theorem \ref{setcoume}) that the space of ends $E$ of any coarse structure with bounded geometry satisfies the first equivalence of Theorem \ref{setcoume}: $E$ is not metrizable $\Leftrightarrow$ $E$ has a continuous surjective map onto the Stone-\v Cech compactification of $\N$.

We refer to Protasov's paper for the relevant notions of coarse geometry.
\end{rem}


\section{The alternative locally finite / separable}\label{s_cell}


\begin{lem}\label{infcysep}
Let $G$ be a group, and $H$ a finitely generated subgroup. Let $X$ be an infinite $G$-set with only finitely many finite $H$-orbits (i.e., the union of infinite $H$-orbits has a finite complement). 

Then the space of ends $\mathcal{E}^\star_G(X)$ of the $G$-set $X$ has a dense subset of cardinal $\le|X|$; in particular its cellularity is $\le |X|$.
\end{lem}
\begin{proof}
Since there is a continuous surjective map $\mathcal{E}^\star_H(X)\to \mathcal{E}^\star_G(X)$, it is enough to show that $\mathcal{E}^\star_H(X)$ has a dense subset of cardinal $\le|X|$ (so $G$ will not appear again in the proof; in particular we can suppose $G=H$).

If $X$ is finite the result is trivial; otherwise $X$ is infinite. If $X$ has only finitely many $H$-orbits, then $|X|=\aleph_0$, and $\mathcal{E}^\star_H(X)$ is a metrizable Stone space, hence is separable.

Now assume that $X$ has infinitely many $H$-orbits, written $(X_i)_{i\in I}$; then each $\mathcal{E}^\star_H(X_i)$ is separable, and hence the union $U=\bigcup_i\mathcal{E}^\star_H(X_i)$ has a dense subset of cardinal $\le |X|$.

To complete the proof, it is therefore enough to show that $U$ is dense, and this is where the assumption that each $X_i$ is infinite (up to finitely many many exceptions) will play a role. Indeed, the density of $U$ is equivalent to the statement that whenever $Y$ is an infinite commensurated subset of $X$, its closure in the end compactification has a nonempty intersection with $U$. This amounts to showing that $Y_i=Y\cap X_i$ is infinite for some $i$. Suppose by contradiction the contrary.

Indeed, for every $s\in H$, the symmetric difference $sY\bigtriangleup Y$ is finite, and hence its union when $s$ ranges over some finite generating subset of $H$ is still finite. Hence, for all but finitely many $i$, $Y_i$ is $H$-invariant, hence empty or equal to $X_i$. Excluding those finitely many finite orbits, this means that for all but finitely many $i$, $Y_i$ is empty. Since $Y_i$ is finite for all $i$, this implies that $Y$ is finite, a contradiction.
\end{proof}


\begin{thm}\label{densalpha}
Let $G$ be a group of infinite cardinal $\alpha$, that is not locally finite. Then the space of ends $\mathcal{E}^\star(G)$ has a dense subset of cardinal $\le\alpha$.
\end{thm}
\begin{proof}
Let $H$ be an infinite, finitely generated subgroup. Since $H$ has only infinite orbits for its left action on $G$, Lemma \ref{infcysep} applies and hence $\mathcal{E}^\star(G)$ has a dense subset of cardinal $\le\alpha$.
\end{proof}




We now prove Theorem \ref{count_cell}. The main implication will be a particular case 
of Theorem \ref{densalpha}.

\begin{thm}\label{count_cell2}
Let $G$ be a countable group. The following are equivalent.
\begin{enumerate}
\item\label{icell_1} $G$ is infinite locally finite;
\item\label{icell_2} $G$ has a left almost invariant map onto $\N$ with finite fibers;
\item\label{icell_3} $\mathcal{E}^\star(G)$ has a continuous map onto the Stone-\v Cech boundary of $\N$;
\item\label{icell_4} $\mathcal{E}^\star(G)$ has a family of $2^{\aleph_0}$ pairwise disjoint open subsets;
\item\label{icell_5} $\mathcal{E}^\star(G)$ does not have countable cellularity;
\item\label{icell_6} $\mathcal{E}^\star(G)$ is not separable.
\end{enumerate}
\end{thm}
\begin{proof}
(\ref{icell_1}) implies (\ref{icell_2}): see Proposition \ref{notame} for this crucial but easy implication, essentially due to Scott and Sonneborn.

(\ref{icell_2}) implies (\ref{icell_3}): immediate by Stone duality.

(\ref{icell_3}) implies (\ref{icell_4}): it is enough to find such a family in the Stone-\v Cech boundary of $\N$: this is a well-known immediate consequence of the existence of $2^{\aleph_0}$ infinite subsets of $\N$ with pairwise finite intersection.

(\ref{icell_4})$\Rightarrow$(\ref{icell_5})$\Rightarrow$(\ref{icell_6}) is trivial.

We finish less trivial implication, namely (\ref{icell_6}) implies (\ref{icell_1}): by contraposition it follows from Theorem \ref{densalpha}.
\end{proof}



Here is a generalization of the cellularity statement of Lemma \ref{infcysep}.

\begin{prop}\label{inforcoce}
Let $H$ be a finitely generated group and $X$ an $H$-set of cardinal $\alpha$ with only infinite orbits. Then the space of ends $\mathcal{E}^\star_H(X)$ of the $H$-set $X$ has cellularity $\le\alpha$.
\end{prop}
\begin{proof}
This is trivial if $\alpha=0$, hence we suppose $\alpha>0$, and hence $\alpha$ is infinite.

By contradiction, suppose that the conclusion fails. Then $\mathcal{E}^\star_H(X)$ has a family of $>\alpha$ pairwise disjoint nonempty clopen subsets. This means that there is a family $(X_i)_{i\in I}$ of infinite $G$-commensurated subsets of $X$, with pairwise finite intersection, with $|I|>\alpha$.

By Lemma \ref{actfg}, for each $i\in I$, there exists a finite subset $F_i$ such that $X_i\smallsetminus HF_{i}$ is $H$-invariant. Since the number of finite subsets of $X$ is $\alpha$, there exists a finite subset $F$ of $X$ and an subset $J$ of $I$ with $|J|>\alpha$ such that $F_i=F$ for all $i\in J$. 
For $i\neq j$ in $J$, on the one hand the subset $(X_i\cap X_j)\smallsetminus HF$ is finite, and on the other hand it is $H$-invariant. Hence it is empty. By cardinality, it follows that there exists a subset $K$ of $J$ with $|K|>\alpha$ (so $K$ is uncountable) such that $X_i\subset HF$ for all $i\in K$.  Since $H$ is finitely generated, the number of $H$-commensurated subsets of $HF$ is countable. We reach a contradiction.
\end{proof}


\begin{thm}\label{cell_uncount2}
Let $G$ be an infinitely-ended group of cardinal $\alpha$. Then the cellularity of $\mathcal{E}^\star(G)$ is $\ge \alpha$; actually there is a nonempty clopen subset $U$ of the end compactification of $G$, with nonempty intersection with $\mathcal{E}^\star(G)$, and a subset $T$ of $G$ of cardinal $\alpha$ such that the $Uh$, for $h\in T$, are pairwise disjoint.
\end{thm}%
\begin{proof}
We argue by constructing an infinite subset $M$ of $G$ and a subset $T$ of $G$ of cardinal $\alpha$ such that the $Mh$ for $h\in T$ are pairwise disjoint.

By the Holt/Dicks-Dunwoody generalization of Stallings' theorem \cite[Theorem IV.6.10]{DD}, there exists an inversion-free action of $G$ on an unbounded tree, a single edge orbit, and finite edge stabilizers. We fix an oriented edge $e$, denote by $e'$ the underlying (non-oriented) edge, and $F$ its (finite) stabilizer. Let $S$ be the set of $s\in G$ such that $e$ and $se$ are equal or adjacent. Since there is a single edge orbit, $S$ generates $G$. 

Let $e^\sharp$ be the set of edges that are on the ``forward" side of $e$ (not including $e$ itself). Define $M$ as the set of $g\in G$ such that $g^{-1}e\in e^\sharp$. We claim that $M$ is left $G$-commensurated. It is enough to show that for every $s\in S$, the set $M_s$ of $g\in M$ such that $sg\notin M$ is finite. Indeed, suppose that $g\in M_s$: we have $g^{-1}e\in e^\sharp$ and $g^{-1}s^{-1}e\notin e^\sharp$. Since $g^{-1}e$ and $g^{-1}s^{-1}e$ are adjacent, this implies that $g^{-1}s^{-1}e'=e'$. This means that $sg\in F$, that is, $g\in s^{-1}F$. Hence $M_s\subset s^{-1}F$ is finite.

Clearly, $M$ is infinite. Let $v$ be the origin vertex of $e$, and $H$ the stabilizer of $v$. We claim that for every $h\in H\smallsetminus F$ we have $M\cap Mh=\emptyset$. Indeed, for $g\in Mh$, we have $gh^{-1}\in M$, which means $hg^{-1}e\in e^\sharp$, that is, $g^{-1}e\in h^{-1}e^{\sharp}=f^\sharp$ where $f=h^{-1}e$. Since $h\in F$, $e$ and $f$ are distinct edges with the same origin, and hence $e^\sharp$ and $f^\sharp$ are disjoint. Hence $g^{-1}e\notin e^\sharp$, which means $g\notin M$. 

It follows that for all $h,h'\in H$ such that $Fh\neq Fh'$, we have $Mh$ and $Mh'$ disjoint. Since $H$ has cardinal $\alpha$, so does $F\backslash H$, so there is a subset $T\subset H$ of cardinal $\alpha$ whose projection to $F\backslash H$ is injective. Hence the $Mh$ for $h\in T$ are pairwise disjoint. 
\end{proof}

Combining Theorems \ref{densalpha} and \ref{cell_uncount2} yields:
\begin{cor}
For an infinitely-ended group $G$ that is not (countable locally finite), of cardinal $\alpha$, both the cellularity and the density (smallest cardinal of a dense subset) of $\mathcal{E}^\star(G)$ are equal to $\alpha$.\qed
\end{cor}

\section{Theorem \ref{exis}: both alternatives can hold}\label{s_both}

\begin{prop}\label{coupme}
Let $G$ be the free product of a family $(G_i)_{i\in I}$ of nontrivial groups ($I$ index set with $0\notin I$). Then $G$ has a left $G$-almost invariant map $f$ onto $I\sqcup\{0\}$), namely mapping $1$ to $0$, and any nontrivial element to the last type of letter in its reduced form. In particular, if $I$ is infinite (say of cardinal $\alpha$), $G$ is not tamely-ended, and the end weight of $G$ is $\ge 2^\alpha$, with equality if $G$ has cardinal $\alpha$. For instance, any infinite free group of infinite rank is non-tamely-ended.
\end{prop}
\begin{proof}
Precisely, if $g\neq 1$, there exists a unique $k\ge 1$, elements $i_1,\dots,i_k$ with $i_j\neq i_{j+1}$ for all $j\in\{1,\dots,k-1\}$, and elements $g_j\in G_{i_j}\smallsetminus \{1\}$, such that $g=g_1\dots g_k$; then by definition, $f(g)=i_k$.

Then it is straightforward that for every $i\in I$ and $g\in G_i$, the set of $h$ such that $f(gh)\neq f(h)$ is contained in $\{1,g^{-1}\}$. Since the set of $G$ for which the set of $h$ such that $f(gh)\neq f(h)$ is finite, is a subgroup, it equals $G$. Thus, $f$ is left $G$-almost invariant. It is clearly surjective. 
\end{proof}



\begin{lem}\label{cardbool}
Let $G$ be a group generated by $H\cup S$ for some subgroup $H$ and finite subset $S$. Then the Boolean algebra $\mathcal{P}_{H,(G)}(G)$ of left $H$-invariant commensurated subsets of $G$ has cardinal bounded above by $\max(\aleph_0,|G|)$.
\end{lem}
\begin{proof}
Let $T$ be a symmetric generating subset of $G$. For $U\subset G$, define its boundary $\partial U$ as the set of pairs $(g,g')$ such that $g'g^{-1}\in T$ and such that exactly one of $g,g'$ belongs to $U$.

We claim that the map $U\mapsto\partial U$, from subsets of $G$ with $1\in U$, to subsets of $G^2$, is injective. We check this by constructing a retraction.
Given a subset $K$ of $G^2$, define $V_K$ as the subset of elements of $G$ that are represented by a word $s_1s_2\dots s_k$ with $s_i\in T$ such that the number of $i\in\{1,\dots,k-1\}$ such that $(s_{i+1}\dots s_k,s_i\dots s_k)\in K$ is even. Now let $U$ be a subset of $G$ with $1\in U$, and write $K=\partial_SU$. Then we readily see that $U=V_K$: indeed it suffices to follow the path and count the number of times where we enter or leave $U$.

Now consider $T=H\cup S$. For $U\subset G$, write $\partial_S U$ for those $(g,g')\in\partial U$ such that $g'g^{-1}\in S$. If $U$ is left $G$-commensurated, then since $S$ is finite, we see that $\partial_S U$ is finite. If $U$ is left $H$-invariant, then $\partial U=\partial_S U$. Therefore the above map $\partial$ maps injectively $H$-invariant commensurated subsets of $G$ containing 1 to finite subsets of $G^2$. Hence, if we relax the condition $1\in U$, its nonempty fibers have at most two elements: the unique one containing 1 and its complement. The upper bound on the cardinality immediately follows.
\end{proof}

\begin{lem}\label{1endfini}
Let $G$ be a group with a subgroup $H$ that is 1-ended and not locally finite. Then the Boolean algebra $\mathcal{P}^\star_G(G)$ of left $G$-commensurated subsets modulo finite subsets is canonically isomorphic to the Boolean algebra $\mathcal{P}_{H,(G)}(G)$ of left $H$-invariant left $G$-commensurated subsets of $G$.
\end{lem}
\begin{proof}
Let $U\subset G$ be a commensurated subset. Then $U\cap Hg$ is left $H$-commens\-urated for every right coset $Hg$ of $H$. Since $H$ is 1-ended, $U\cap Hg$ is either finite or cofinite in $Hg$, for every right coset $Hg$. Let $f(U)$ be the union of all right $H$-cosets $Hg$ such that $U\cap Hg$ is cofinite in $U$: then $V=f(U)$ is the unique left $H$-invariant subset $V\subset G$ such that $(V\bigtriangleup U)\cap Hg$ is finite for every right coset $Hg$. Hence $f$ is a homomorphism of Boolean algebras $\mathcal{P}_{(G)}(G)\to\mathcal{P}(G)$.

Since $H$ is not locally finite, it includes an infinite finitely generated subgroup $L$. Then, for every $U\in\mathcal{P}_{(G)}(G)$, the intersection $U\cap Hg$ is $L$-invariant for all but finitely many cosets $Hg$. Since it is finite or cofinite, we deduce that $U\cap Hg\in\{\emptyset,Hg\}$ for all but finitely many cosets $Hg$. This means that $U\tu f(U)$ is finite for all $U$.
It follows in particular that $f(U)$ is $G$-commensurated as well (which was not clear a priori). Then the kernel of $f$ is the ideal of finite subsets of $G$, and its image is equal to $\mathcal{P}_{H,(G)}(G)$.
\end{proof}

\begin{prop}\label{coume}
Let $H$ be any group of cardinal $\alpha$ and $G=H\ast\Z$, where $\Z=\langle t\rangle$. Then the end weight of $G$ (i.e., the cardinal of $\mathcal{P}^\star_G(G)$ or $\mathcal{P}_{(G)}(G)$) has cardinal $\ge\alpha$, with equality if $H$ is 1-ended and not locally finite.

In particular, if $H$ is a 1-ended countable group, then $H\ast\Z$ is metrizably-ended.
\end{prop}
\begin{proof}
The lower bound follows from Theorem \ref{cell_uncount2}; let us however provide a direct argument: for every $h\in H\smallsetminus\{1\}$, the set of elements whose reduced form finishes with $h$ is commensurated, infinite, and these are pairwise disjoint.

For the upper bound, if $H$ is 1-ended and not locally finite, then the cardinal is exactly $\alpha$, by Lemma \ref{1endfini} and Lemma \ref{cardbool}.
\end{proof}

\begin{prop}
Let $H$ be a group in which every countable subgroup is contained in a countable 1-ended subgroup. Then $G=H\ast\Z$ is tamely-ended. In particular, if $\alpha$ is an uncountable cardinal and $H$ is an abelian group of cardinal $\alpha$, then $H\ast\Z$ has cardinal $\alpha$, is tamely-ended but not metrizably-ended.
\end{prop}
\begin{proof}
Given Proposition \ref{coume}, it only remains to show that $G$ is tamely-ended. Indeed, suppose by contradiction that there exists a surjective, almost invariant map $f:G\to\N$. Hence it is surjective in restriction to $L\ast\Z$ for some countable subgroup $L$ of $G$, and by assumption, $L$ is contained in a countable 1-ended subgroup $L'$ of $H$. Hence $L'\ast\Z$ is not tamely-ended, and this contradicts Proposition \ref{coume}.
\end{proof}



Let $f:H\to G$ be a group homomorphism; if it is injective (or more generally has finite kernel), it induces a continuous map $i_*:\mathcal{E}^\star(H)\to \mathcal{E}^\star(G)$. Then $i_*$ often fails to be injective, notably when $G$ is 1-ended and $H$ has $\ge 2$ ends. Here is an injectivity result:

\begin{prop}\label{freepro_inj}
Let $G$ be a free product $H\ast L$ of groups. Then the inclusion $i$ of $H$ into $G$ induces an injective continuous map $i_*:\mathcal{E}^\star(H)\to \mathcal{E}^\star(G)$.
\end{prop}
\begin{proof}

Given an element $g$ of the free product $H\ast L$, write it as a reduced word with respect to $H\cup L$, define its $H$-suffix $\mathrm{suf}_H(g)$ to be $h$ if the last ``letter" in the reduced decomposition of $g$ is $h\in H$, and $1$ otherwise. For instance, $\mathrm{suf}_H(h)=\ell h$ for all $(\ell,h)\in L\times H$ and $\mathrm{suf}_H(h\ell)=1$ for all $h\in H$ and all $\ell\in L\smallsetminus\{1\}$. Then $\mathrm{suf}_H(\ell g)=\mathrm{suf}_H(g)$ for all $\ell\in L$ and $g\in G$, and for $h\in H$, we have $\mathrm{suf}_H(hg)=\mathrm{suf}_H(g)$ for all $g\in G\smallsetminus H$.

Now let us proceed to the proof. By Stone duality, this amounts to checking that the canonical map $i^*:\mathcal{P}^\star_G(G)\to \mathcal{P}^\star_H(H)$ is surjective. This map simply sends a left $G$-commensurated subset of $G$ to its intersection with $H$. Let $M\in\mathcal{P}^\star_H(H)$ be a left $H$-commensurated subset of $H$.
Now define $M'=\{g\in G:\mathrm{suf}_H(g)\in M\}$. Clearly $M'\cap H=M$. Surjectivity is proved if we can check that $M$ is left $G$-commensurated. By the property of $\mathrm{suf}_H$, the subset $M'$ is left $L$-invariant. For $h\in H$, consider $g\in M'\smallsetminus h^{-1}M'$. This means that $\mathrm{suf}_H(g)\in M$ while $\mathrm{suf}_H(hg)\notin M'$. The partial $H$-invariance of $\mathrm{suf}_H$ implies that $g\in H$. Hence $M'\smallsetminus h^{-1}M'=M\smallsetminus h^{-1}M$ is finite for all $h\in H$.
\end{proof}

\section{No isolated point: proof of Theorem \ref{noisola}}\label{s_nois}

Theorem \ref{noisola} asserts that for every infinitely-ended group $G$, the space of ends of $G$ has no isolated point.
The locally finite case is covered by Holt's theorem, which, stated this way, says that $G$ is countable, and by Protasov's theorem.

Now suppose that $G$ is not locally finite; let $H$ be an infinite, finitely generated subgroup. Let us then prove a stronger result, generalizing Abels' Theorem \ref{abels75}:

\begin{prop}\label{isolaspecker}
Let $G$ be a infinitely-ended group that is not locally finite. Then every $G$-equivariant Stone space quotient of $\mathcal{E}^\star(G)$ is perfect or reduced to a singleton. In other words, every right $G$-invariant Boolean subalgebra of $\mathcal{P}^\star_G(G)$, not reduced to $\{0,1\}$, is non-atomic. (In the language of \cite{Abe}, this says that the boundary of every Specker compactification of $G$, if not reduced to a singleton, is a perfect space.)
\end{prop}

\begin{proof}
Let $A'$ be a right $G$-invariant subalgebra of $\mathcal{P}^\star_G(G)$, not reduced to $\{0,1\}$; suppose by contradiction that $A'$ is atomic. Let $A\subset\mathcal{P}_{(G)}(G)$ be its inverse image in $G$.


The assumption means that there exists an infinite left $G$-commensurated subset $X\subset G$, with $X\in A$ such that for every $Y\in A$, either $Y$ is near disjoint from $X$ (in the sense that $X\cap Y$ is finite), or $X$ is near contained in $Y$ (in the sense that $X\smallsetminus Y$ is finite). 

Since $G$ is non-locally-finite with $\ge 3$ ends, by Abels' result (Theorem \ref{abels75}), $A$ is infinite. Hence, there exists infinite left $G$-commensurated subsets $Z_1,Z_2\in A$, such that $X$, $Z_1$ and $Z_2$ are pairwise disjoint.


We claim that there exists a finitely generated subgroup $L$ of $G$ such that $X'=L\cap X$, $Z'_1=L\cap Z_1$ and $Z'_1=L\cap Z_2$ are all infinite. There exists at least one right coset $Hg_0$ such that $X\cap Hg_0$ is infinite: indeed, otherwise, this intersection is finite for all $g$, and then by Lemma \ref{actfg} and using that $H$ is infinite, it is empty for all but finitely many $Hg$, which would imply that $X$ is finite. Similarly, for $i=1,2$, at least one right coset $Hg_i$ such that $Z_i\cap Hg_i$ is infinite. Then we can choose $L=\langle H\cup\{g_0,g_1,g_2\}\rangle$ and the claim is proved.

Let $L$ be as given by the claim. Then $L$ has at least 3 ends, and $X'$ defines a nonempty, proper subset $M$ of $\mathcal{E}^\star(L)$. Since $A$ is right $G$-invariant, the previous property of $X$ applies to $Y=Xh$ for $h\in L$, and yields: for every $h\in L$, $X'$ is either near contained in $X'h$, or near disjoint to $X'h$. This implies that for every $h\in L$, we have $Mh$ either contained or disjoint to $M$. Applying this to $h^{-1}$, we deduce the stronger fact that for every $h\in L$, we have $Mh$ either equal or disjoint to $M$. The complement of $\bigcup_{h\in L}Mh$ is a proper closed $L$-invariant subset of $\mathcal{E}^\star(L)$; by minimality of the $L$-action, it is empty. It follows that the union is finite, that is, there exists a finite index subgroup of $L$ preserving $M$. Since the minimality of the action on the space of ends holds for every finite index subgroup of $L$, we deduce that $M=\mathcal{E}^\star(L)$. Since both $X'$ and $L\smallsetminus X'$ are infinite, we reach a contradiction.
\end{proof}

\section{Bi-ends: discussion and proof of Theorem \ref{tbiends}}\label{biends}

In \cite{ScS}, 1-ended is called ``completely regular" and 1-bi-ended is called ``regular". They prove that
\begin{itemize}
\item every group with an infinite, infinite index finitely generated subgroup, is 1-bi-ended. 
\item every uncountable group is 1-bi-ended.
\end{itemize}

This obviously applies to infinitely generated groups that are not (countable locally finite). For finitely generated groups, this discards virtually cyclic groups on the one hand (for which they provide a complete discussion), and others. Others that are not covered are (non-virtually cyclic) finitely generated groups in which every infinite index subgroup is locally finite; note that such groups are necessarily torsion. It turns out that there exist such groups (later discovered: Tarski monsters). However, they are 1-ended. Indeed, it is an immediate consequence of Stallings' theorem that finitely generated groups with $\ge 2$ ends are not torsion. Nevertheless, let us mention that this already follows from Freudenthal's remarkable work \cite{Fr2}, where he obtained the following, which immediately implies that every infinitely-ended group is non-torsion.

\begin{thm}[Freudenthal, 7.6 in \cite{Fr2}]\label{frenontor}
Every finitely generated group with at least 3 ends has a nonempty subset $Q$ and an element $g$ such that $gQ\subset Q$ and $\bigcap_{n}g^nQ=\emptyset$.
\end{thm}

While this is now superseded by Stallings' theorem, the method, which actually shows that the action on the boundary is a convergence action (as later formalized by Gehring and Martin \cite{GM87}) is not. 

Let us now pass to the proof of Theorem \ref{tbiends}.
(\ref{scs1}) is trivial, (\ref{scs2}) is essentially \cite[Theorem 3(2)]{ScS}; (\ref{scs3}) is Proposition \ref{notame}, which follows the argument in \cite[Theorem 4]{ScS}. In (\ref{scs4}), let us first treat the case when $G$ is finite-by-$D_\infty$: in this case \cite[Theorem 3(1)]{ScS} states that $G$ is 1-bi-ended. Now the main lemma in \cite{ScS} is that if $G$ has an infinite subgroup $H$ of index $|G|$, generated by a subset of cardinal $<|G|$, then $G$ is 1-bi-ended. This immediately applies if $G$ is uncountable (\cite[Theorem 1]{ScS}), and also if it has an infinite finitely generated subgroup of infinite index. This applies when $G$ is infinitely generated and not locally finite. This boils down to the finitely generated case. 

Clearly $G$ is 1-bi-ended when $G$ is 1-ended (an at most 1-ended group is called ``completely regular" in \cite{ScS}). The virtually-$\Z$ case being settled, it remains to consider finitely generated groups that are not virtually-$\Z$ and have at least 2 ends. That 2-ended groups are virtually cyclic is independently due to Hopf and Freudenthal, so it remains to consider the case with at least 3 ends. By Freudenthal's theorem \ref{frenontor}, this implies that the group has an infinite cyclic subgroup, necessarily of infinite index, and hence the Scott-Sonneborn result applies.

The locally compact version of the theorem holds with the more-or-less obvious natural changes: finite $\to$ compact; finite-by-$\Z$ $\to$ compact-by-($\Z$ or $\R$); countable $\to$ $\sigma$-compact; locally finite $\to$ regionally elliptic; finitely generated $\to$ compactly generated. Since no new specific phenomenon occurs, we omit the proof.

Let us also mention that the result that infinite countable locally compact groups are untamely-biended is implicit in the proof of \cite[Theorem 4]{ScS}.

\begin{prop}\label{notame}
Let $G$ be an infinite countable locally compact group. Then $G$ is untamely-biended (and hence not tamely-ended). That is, there exists a function $\ell:G\to\N$, with infinite image, that is almost $G$-bi-invariant, in the sense that for every $g\in G$, the set of $h\in G$ such that $\ell(gh)=\ell(hg)=\ell(h)$ has a finite complement. Moreover, $\ell$ can be chosen with finite fibers.
\end{prop}
\begin{proof}
Write $G=\bigcup G_n$, with $G_n$ finite and $G_n\subset G_{n+1}$. Write $\ell(g)=\min\{n:g\in G_n\}$. Then $\ell$ is almost $G$-bi-invariant: indeed, for every given $g$, for every $h$ such that $\ell(h)>\ell(g)$ we have $f(gh)=f(hg)=f(h)$. (The proof of \cite[Theorem 4]{ScS} consists in observing that $\ell^{-1}(2\N)$ is bi-$G$-commensurated, but this obviously holds for every subset of $\N$ in lieu of $2\N$.)
\end{proof}

\section{Focal groups}

By TDLC-group, we mean totally disconnected, locally compact group.
A focal TDLC group is by definition a strictly HNN extension of a profinite group. A locally focal TDLC-group is a TDLC-group in which there exists a compactly generated open subgroup $H$, such that every compactly generated subgroup of $G$ containing $H$ is focal. (Thus locally focal TDLC-group is focal if and only if it is compactly generated.)

\begin{prop}\label{locfocal}
Every non-focal, locally focal TDLC-group is 1-ended.
\end{prop}

\begin{lem}\label{focalopen}
Let $G$ be a focal TDLC-group. Then every open subgroup of $G$ has finite index, or is regionally elliptic.
\end{lem}
\begin{proof}
Let $N$ be the regionally elliptic radical of $G$, so $G/N$ is isomorphic to $\mathbf{Z}$. Let $H$ be an open subgroup that is not regionally elliptic. Then $H$ is not contained in $N$, so $HN$ has finite index, so we have to prove that $H$ has finite index in $HN$; we can now suppose that $HN=G$. Let $u$ be an element of $H$ projecting to a generator of $G/N$, such that $uKu^{-1}\subset K$ for some compact open subgroup $K$ of $N$ and $\bigcup_n u^{-n}Ku^n=N$. Let $f$ be the (compacting) automorphism of $N$ given by conjugation by $u$, so $f(K)\subset K$, and $f(H)=H$. 

Let $d$ be the index of $H\cap K$ in $K$. We claim that $H$ has index $\le d$: equivalently, we have to show that $N\cap H$ has index $\le d$ in $N$. Indeed, let $x_0,\dots,x_d$ be elements of $N$. So there exists $n$ such that all $x_i$ belong to $f^{-n}(K)$. Then $f^n(x_i)\in K$ for all $i$, and hence there exists $i\neq j$ such that $f^n(x_i)^{-1}f^n(x_j)\in H$. Hence $x_i^{-1}x_j\in f^{-n}(H)=H$. This shows that $N\cap H$ has index $\le d$.
\end{proof}

Let $H$ be an open subgroup of an LC-group. We say that $H$ is hardly normal if for every $g\notin H$, we have $gHg^{-1}\cap H$ non-compact.


\begin{lem}\label{lemhama}
Let $G$ be an LC-group that is not compactly generated, with a compactly generated open subgroup that is hardly normal. Then $G$ is 1-ended.
\end{lem}
\begin{proof}
Let $X$ be a $G$-commensurated subset of $G$. Then $X\cap Hg$ belongs to $\{\emptyset,Hg\}$ for all but finitely many $Hg$. Suppose by contradiction that there exist $g_1,g_2$ such that $X\cap Hg_1=\emptyset$ and $X\cap Hg_2=Hg_2$. 

We have $g_2g_1^{-1}Hg_1\cap Hg_2=(g_2g_1^{-1}Hg_1g_2^{-1}\cap H)g_2$, which by assumption is closed non-compact, and contained in $X$. Since $Hg_1\cap X$ is empty and $g_2g_1^{-1}X\tu X$ has compact closure, the subset $g_2g_1^{-1}Hg_1\cap X$ has compact closure. This is a contradiction.

Hence we have proved that for every left $G$-commensurated subset $X$, either $X$ or its complement is contained in  $HF$ for some finite subset $F$. Let $X$ be such a subset.

Let $g$ be an element of $G$ not contained in the compactly generated subgroup $\langle H\cup F\rangle$. Then $gHF\cap HF$ is empty. Hence $gX\cap X$ is empty. Since $X$ is left commensurated, this implies that $X$ is compact.
\end{proof}

\begin{proof}[Proof of Proposition \ref{locfocal}]
Let $H$ be as in the definition of locally focal group. For any $g\in G$, let $H'$ be the subgroup generated by $H$ and $g$: it is focal, and hence, by Lemma \ref{focalopen}, $H$ has finite index in $H'$. In particular, $H$ and $gHg^{-1}$ have a common open subgroup of finite index. Thus $H$ is hardly normal (actually it is a --groupwise-- commensurated subgroup). By Lemma \ref{lemhama}, $G$ is 1-ended.
\end{proof}

\section{Proofs for locally compact groups}

First, if $G$ is a locally compact group whose unit component $G_0$ is non-compact, then $G$ has at most 2 ends (\cite[Theorem 4.2]{Hou} or \cite[Satz 7.4]{Abe}). Hence we can suppose that $G_0$ is compact; we then easily check that $G_0$ acts trivially on the space of ends, which is then the same for $G$ and $G/G_0$. Accordingly, we can suppose that $G$ is totally disconnected.

The proof of Theorem \ref{noisola2} separately consists in the regionally elliptic case, and otherwise. In the non-locally-elliptic case, the analogue of Proposition \ref{isolaspecker} holds. The proof is an adaptation of that of Proposition \ref{isolaspecker}, with an issue related to focal groups. Recall that the first step consists in finding an open, compactly generated subgroup $L$ with at least 3 ends (and such that every overgroup of $L$ in $G$ has at least 3 ends). The proof of the existence of $L$ immediately adapts. However, for the sequel of the proof, we need to ensure that $L$ is not focal. At this point, we use Proposition \ref{locfocal}, which says that $G$ has an open subgroup, containing $L$, that is not focal. The remainder of the proof now adapts.

For $G$ locally compact, say that $G$ is untamely-ended if there is a surjective continuous function $G\to\N$, such that for every compact subset $K$ of $G$, the set of $g\in G$ such that $f(g)\neq f(hg)$ for some $h\in K$, has compact closure. Otherwise, say that $G$ is tamely-ended.

\begin{thm}
Let $G$ be a $\sigma$-compact LC-group. The following are equivalent:
\begin{enumerate}
\item\label{tit_mt1} $G$ is metrizably-ended;
\item\label{tit_mt2} the space of ends of $G$ has no continuous surjective map onto the Stone-\v Cech compactification of $\N$;
\item\label{tit_mt3} the end compactification of $X$ has no continuous surjective map onto the Stone-\v Cech compactification of $\N$;
\item\label{tit_mt4} $G$ is tamely-ended.
\item\label{tit_mt5} there exists an open, compactly generated subgroup $H$ of $G$ such that the only topologically left $G$-commensurated, $H$-invariant subsets of $G$ are $\emptyset$ and $G$.
\end{enumerate}
\end{thm}
The proof is a straightforward adapatation of that in the discrete case (Theorem \ref{setcoume}), so we omit it.


\begin{thm}\label{count_celltop}
Let $G$ be a $\sigma$-compact LC-group. The following are equivalent.
\begin{enumerate}
\item\label{noncle} $G$ is non-compact and regionally elliptic;
\item the space of ends $\mathcal{E}^\star(G)$ has a continuous map onto the Stone-\v Cech boundary of $\N$;
\item $\mathcal{E}^\star(G)$ has a family of $2^{\aleph_0}$ pairwise disjoint nonempty open subsets.
\item\label{coucel} $\mathcal{E}^\star(G)$ does not have countable cellularity;
\end{enumerate}
\end{thm}
Again, the proof of Theorem \ref{count_cell} adapts with routine changes and we omit it.




\section{Holt's theorem in the locally compact setting}\label{s_holt}

The following is an extension of a lemma of Holt:

\begin{lem}\label{holtlemma}
Let $G$ be a group and $H,L$ two subgroups such that
\begin{enumerate}
\item $H\cup L$ generates $G$, and
\item\label{reltor} for every $g\in G$ there exists $n\ge 1$ such that $g^n\in H\cap L$.
\end{enumerate}
Let $I$ be a set and $f:G\to I$ be a function that is
constant on each right coset $Hg$, $g\notin H$ and on each right coset $Lg$, $g\notin L$. Then $f$ is constant on $G\smallsetminus (H\cap L)$. If moreover $f$ is constant on $L$, then it is constant on $G$.
\end{lem}
\begin{proof}
The same lemma is asserted and proved in \cite[\S 3]{Hol}, replacing (\ref{reltor}) by the stronger assumption that $G$ is a torsion group. The latter hypothesis is used only in the very beginning of the proof, where it is clear that (\ref{reltor}) is enough, and the remainder of the proof works with no further change.
\end{proof}

\begin{proof}[Proof of Theorem \ref{HLC}]
Let us show that every non-$\sigma$-compact, regionally elliptic LC-group $G$ is 1-ended.

Fix a compact open subgroup $K$ of $G$. Let $f$ be a function from $G$ to $\Z/2\Z$ that is right $K$-invariant, and such that for every $h\in G$, the function $f_h:g\mapsto f(hg)-f(g)$ has compact support $U_h$, uniformly on compact subsets, in the sense that $\bigcup_{h\in W}U_h$ is compact for every compact subset $W$ of $G$ (note that this union is right $K$-invariant, so is automatically a clopen subset). We have to show that $f$ is constant outside some compact subset. Denote by $f'$ the factored function on $G/K$.

For $N$ a subgroup of $G$, define $q(N)$ as the subgroup generated by $N\cup \bigcup_{g\in N}U_g$. By definition $N\subset q(N)$. If moreover $N$ is open and $\sigma$-compact, then so is $q(N)$.

Then we fix a non-compact, $\sigma$-compact open subgroup $H_0$ containing $K$, and define $H_{n+1}=q(H_n)$, and $H=\bigcup_{n\ge 0}H_n$. Hence $H$ is a non-compact open $\sigma$-compact subgroup containing $K$, and moreover $q(H)=H$. Hence $f$ is constant on all right cosets $Hg$, $g\notin H$. Since $G$ is not $\sigma$-compact, we have $H\neq G$.


Since $G$ is regionally elliptic and $G\neq H$, there is a compact open subgroup $M$ such that $K\subset M$ and $M$ is not contained in $H$. Define $G'$ as the subgroup generated by $H\cup M$; both $H$ and $M$ are proper subgroups of $G'$.

By Lemma \ref{actfg2}, there is a finite subset $F$ of $G'$ such that $f$ is left $M$-invariant outside $MF$. Then $F'$ is contained in the subgroup generated by $F'\cup M$ for some finite subset $F'$ of $H$; in addition we can suppose that $F'$ is not contained in $M$. Let $L$ be the subgroup generated by $F'\cup M$; then $L$ is a compact open subgroup. By definition, $L$ is generated by $H\cap L$ and $M$, and both are proper subgroups of $L$.




We first claim that for every $g\in G\smallsetminus(LH\smallsetminus H)$, then $f$ is constant on $(H\cap L)\ell g$ for every $\ell\in L\smallsetminus H$. Indeed, for such $g,\ell$, we have $\ell g\notin H$: indeed otherwise, we have $g\in LH$, which forces $g\in H$ by the assumption of the claim, and then $\ell g\in H$ implies $\ell\in H$, a contradiction. Hence $\ell g\notin H$, and thus $f$ is constant on $H\ell g$, and the result follows.

Second, we claim that $f$ is constant on $Lg$ for every $g\in G'\smallsetminus L$. For this, it does not matter if we left-multiply $g$ by some element of $L$, so we can suppose that $g\notin LH\smallsetminus H$. Hence, by the first claim, $f$ is constant on $(H\cap L)\ell g$ for every $\ell\in L\smallsetminus H$. Define $u:L\to\Z/2\Z$ by $u(\ell)=f(\ell g)$. Then this restates as: $u$ is constant on $(H\cap L)\ell$ for every $\ell\in L\smallsetminus (H\cap L)$. In addition, $u$ is left $M$-invariant, since $f$ is left $M$-invariant on $G'\smallsetminus L$. Hence, by the second assertion of Lemma \ref{holtlemma} (applied to the subgroups $H\cap L$ and $M$ of $L$), $u$ is constant on $L$. This means that $f$ is constant on $Lg$.


Since $f$ is constant on $Hg$ for every $g\in G'\smallsetminus H$, Lemma \ref{holtlemma} (applied to the subgroups $H$ and $L$ of $G'$) implies that $f$ is constant on $G'\smallsetminus (H\cap L)$. (Keep in mind that $G'$ and $L$ depend on $M$.)

In particular, we have shown that $f$ is constant on $M\smallsetminus H$ for every compact open subgroup $M$ not contained in $H$. Since every pair of elements of $G\smallsetminus H$ is contained in $M$ for some choice of $M$, this shows that $f$ is constant outside $H$, say equal to $c$. Next, fix a single $k\notin H$: then for every $h\in H\smallsetminus U_k$, we have $c=f(kh)=f(h)$. Hence $f$ equals $c$ outside $U_k$.
\end{proof}

\noindent {\bf Acknowledgements.} I thank Pierre-Emmanuel Caprace and Pierre de la Harpe for various useful corrections and remarks. I thank Igor Protasov for pointing out \cite{BCZ}. I am indebted to the referee for valuable corrections and remarks, notably a suggestion which led to Proposition \ref{free_prod_inj_intro}.







\end{document}